\documentclass[12pt,reqno,oneside]{amsart}
\usepackage[headheight=15pt,rmargin=0.5in,lmargin=0.5in,tmargin=0.75in,bmargin=0.75in]{geometry}
\usepackage{imakeidx}
\usepackage{framed}
\usepackage{amsthm}
\usepackage{amssymb}
\usepackage{amsmath}
\usepackage{mathrsfs}
\usepackage{enumitem}
\usepackage[hyphens,spaces,obeyspaces]{url}

\usepackage{array,makecell}
\usepackage[capitalise,noabbrev]{cleveref}
\usepackage{appendix}
\usepackage{tikz}
\usepackage{tikz-cd}
\usepackage{nomencl}\makenomenclature
\usetikzlibrary{braids,arrows,decorations.markings}

\newtheorem{theorem}{Theorem}[section]

\newtheorem{corollary}{Corollary}[section]
\newtheorem{lemma}{Lemma}[section]

\newtheorem{remark}{Remark}[section]

\title{A Formula for the Number of $k$-almost Primes}
\author{Steven Creech}
\date{October 2023}

\begin{document}

\begin{abstract}
    In 2005, E. Noel and G. Panos constructed a formula to count the number of semiprimes less than a given value $x$. In 2006, this formula was rediscovered independently by R.G. Wilson V. However, each citation of this result was simply through personal communication, and a formal proof has not been written down anywhere in the literature. In this article, we generalize their formula to the case of $k$-almost primes and square-free $k$-almost primes.
\end{abstract}

\maketitle

\section{Introduction}

Given a positive integer $n>1$, we have that $n$ has a prime factorization given by 
\[
n=\prod_{i=1}^r p_i^{a_i}.
\]
We recall that the function $\Omega(n):=\sum_{i=1}^r a_i$ counts the number of prime factors of $n$ (with multiplicity). We have that a $k$-almost prime is a number $n$ such that $\Omega(n)=k$. We note that a $1$-almost prime is just a prime number while a $2$-almost prime is known as a semiprime. Now we wish to have a generalization of the prime counting function, $\pi(x)=\#\{n\leq x: n\text{ is prime}\}$ to the $k$-almost prime counting function 
\[
\pi_k(x)=\#\{n\leq x: \Omega(n)=k\}.
\]
We will also define the square-free $k$-almost prime counting function as 
\[
\pi_k^*(x)=\#\{n\leq x: \Omega(n)=k, n\text{ is square-free}\}.
\]

Now it was remarked in \cite{weissteinSemiprime} that the following formula for $\pi_2(x)$ was discovered in 2005 by E. Noel and G. Panos, and then independently rediscovered by R.G. Wilson V. in 2006: 
\begin{equation}\label{eq:semiprime}
\pi_2(x)=\sum_{i=1}^{\pi(\sqrt{x})}\left(\pi\left(\frac{x}{p_i}\right)-i+1\right).
\end{equation}
Where in the above formula we have that $p_i$ denotes the $i$th prime number. 
However, the citation for both discoveries was given as personal communications and a formal proof of this formula does not exist anywhere in the literature. Thus, we will finally give a proof of this formula. In addition, we will show that with a minor modification we get the following formula for the number of square-free semiprimes:
\begin{equation}\label{eq:sqfr_semiprime}
\pi_2^*(x)=\sum_{i=1}^{\pi(\sqrt{x})}\left(\pi\left(\frac{x}{p_i}\right)-i\right).
\end{equation}

Furthermore, we will generalize the above formulas to the case of $k$-almost primes. We summarize this into the following theorem.

\begin{theorem}\label{thm:MainThm}
    For $k\geq 2$, we have the following formula for the number of $k$-almost primes less than $x$ 
    \begin{equation}\label{eq:kalmostprime}
        \pi_k(x)=\sum_{i_{k-1}=1}^{\pi(\sqrt[k]{x})}\sum_{i_{k-2}=i_{k-1}}^{\pi\left(\sqrt[k-1]{\frac{x}{p_{i_{k-1}}}}\right)}\sum_{i_{k-3}=i_{k-2}}^{\pi\left(\sqrt[k-2]{\frac{x}{p_{i_{k-2}}p_{i_{k-1}}}}\right)}...\sum_{i_1=i_2}^{\pi\left(\sqrt{\frac{x}{p_{i_2}p_{i_3}...p_{i_{k-1}}}}\right)}\left(\pi\left(\frac{x}{p_1p_2...p_{k-1}}\right)-i_1+1\right)
    \end{equation}
    Similarly, we have the following formula for the number of square-free $k$-almost primes less than $x$
    \begin{equation}\label{eq:sqfrkalmostprime}
        \pi_k^*(x)=\sum_{i_{k-1}=1}^{\pi(\sqrt[k]{x})}\sum_{i_{k-2}=i_{k-1}+1}^{\pi\left(\sqrt[k-1]{\frac{x}{p_{i_{k-1}}}}\right)}\sum_{i_{k-3}=i_{k-2}+1}^{\pi\left(\sqrt[k-2]{\frac{x}{p_{i_{k-2}}p_{i_{k-1}}}}\right)}...\sum_{i_1=i_2+1}^{\pi\left(\sqrt{\frac{x}{p_{i_2}p_{i_3}...p_{i_{k-1}}}}\right)}\left(\pi\left(\frac{x}{p_1p_2...p_{k-1}}\right)-i_1\right)
    \end{equation}
\end{theorem}

\begin{remark}
    A sphenic number is a squre-free $3$-almost prime. I had found the StackExchange post \cite{1140945} which gave the formula for $\pi_3^*(x)$ in the above; however, it did not provide a proof of the formula. However, it should be noted that this post served as an inspiration to generalize to the $k$-almost prime case.
\end{remark}

Now this paper will be organized into two parts. In the first part, we shall just go through the semiprime case, and then we shall move to the $k$-almost prime case. We separate out the semiprime case for two reasons. Namely, for the $k$-almost prime case, the proof is done by induction on $k$, and the semiprime case will be the base case. Secondly, since this formula for $\pi_2(x)$ has been written down, it is nice to have an explicit proof for this specific formula. 

\section{Proof of the Semiprime Case}
In this section, we shall give a proof of \cref{eq:semiprime} and \cref{eq:sqfr_semiprime}. 

\begin{lemma}\label{lem:semiprimelemma}
    The number of semiprimes less than or equal to a value $x$ whose smallest prime factor is $p_i$ is 
    \[
    \pi\left(\frac{x}{p_i}\right)-i+1.
    \]
    Furthermore, the number of square-free semiprimes less than or equal to $x$ whose smallest prime factor is $p_i$ is
    \[
    \pi\left(\frac{x}{p_i}\right)-i.
    \]
\end{lemma}

\begin{proof}
    Given a semiprime $n\leq x$ whose smallest prime factor is $p_i$, we can then write $n=p_iq$ where $p_i\leq q$, and $q$ is a prime number. Thus, we have that $p_i\leq q\leq \frac{x}{p_i}$. Thus, there are $\pi\left(\frac{x}{p_i}\right)-\pi(p_i-1)=\pi\left(\frac{x}{p_i}\right)-i+1$ choices for the prime $q$, and we conclude the formula. 

    We remark that the same proof works for the semiprime case, but we must have that $p_i<q\leq \frac{x}{p_i}$, so we will have $\pi\left(\frac{x}{p_i}\right)-\pi(p_i)=\pi\left(\frac{x}{p_i}\right)-i$ choices for $q$.
\end{proof}

Now to derive \cref{eq:semiprime} and \cref{eq:sqfr_semiprime} we begin by observing that the largest possible prime factor for a semiprime less than or equal to $x$ is $\sqrt{x}$. We note that this will actually be achieved if and only if $x=p^2$ for some prime. Thus, the first $\pi(\sqrt{x})$ prime numbers will be the only possible prime factors for semiprimes less than or equal to $x$. Now to count all the (square-free) semiprimes less than or equal to $x$, we count the number of (square-free) semiprimes less than or equal to $x$ whose smallest prime factor is $p_i$, and then add up all of these semiprimes for each of the $\pi(\sqrt{x})$ possible primes. These sums are precisely \cref{eq:semiprime} and \cref{eq:sqfr_semiprime}.

We remark that if we start the sums at some value $n$ instead of $1$, we that will count the number of (square-free) semiprimes less than $x$ whose smallest prime factor is at least $p_n$ (although it may be greater). Furthermore, we see that setting $n=1$ in the below formula will have us recover \cref{eq:semiprime} and \cref{eq:sqfr_semiprime}.

\begin{corollary}\label{cor:semiprimecor}
    The number of semiprimes less than $x$ whose smallest prime factor is at least $p_n$ is
    \[
    \sum_{i=n}^{\pi(\sqrt{x})}\pi\left(\frac{x}{p_i}\right)-i+1.
    \]
    Similarly, for the number of square-free semiprimes less than $x$ whose smallest prime factor is at least $p_n$ is 
    \[
    \sum_{i=n}^{\pi(\sqrt{x})}\pi\left(\frac{x}{p_i}\right)-i.
    \]
    
\end{corollary}

\section{Proof for the General $k$-almost Prime Case}

We shall now give a proof of \cref{thm:MainThm}. We shall do this by a similar counting argument. The main observation is that a $k$-almost prime is the product of a prime and a $k-1$-almost prime which will allow us to induct. Thus, we will have the following lemma which is a generalization of \cref{cor:semiprimecor} for the $k$-almost prime case. We will also remark that the proof of \cref{thm:MainThm} follows directly from \cref{Lem:kalmostlemma} by setting $n=1$.

\begin{lemma}\label{Lem:kalmostlemma}
    The number of $k$-almost primes less than a given value $x$ whose smallest possible prime factor is $p_n$ is 
    \begin{equation}\label{eq: kalmostprimes nversion}
    \sum_{i_{k-1}=n}^{\pi(\sqrt[k]{x})}\sum_{i_{k-2}=i_{k-1}}^{\pi\left(\sqrt[k-1]{\frac{x}{p_{i_{k-1}}}}\right)}\sum_{i_{k-3}=i_{k-2}}^{\pi\left(\sqrt[k-2]{\frac{x}{p_{i_{k-2}}p_{i_{k-1}}}}\right)}...\sum_{i_1=i_2}^{\pi\left(\sqrt{\frac{x}{p_{i_2}p_{i_3}...p_{i_{k-1}}}}\right)}\left(\pi\left(\frac{x}{p_1p_2...p_{k-1}}\right)-i_1+1\right)
    \end{equation}

    Similarly, the number of (square-free) $k$-almost primes less than a given value $x$ whose smallest possible prime factor is $p_n$ is 

    \begin{equation}\label{eq: sqfrkalmostprime nversion}
    \sum_{i_{k-1}=n}^{\pi(\sqrt[k]{x})}\sum_{i_{k-2}=i_{k-1}+1}^{\pi\left(\sqrt[k-1]{\frac{x}{p_{i_{k-1}}}}\right)}\sum_{i_{k-3}=i_{k-2}+1}^{\pi\left(\sqrt[k-2]{\frac{x}{p_{i_{k-2}}p_{i_{k-1}}}}\right)}...\sum_{i_1=i_2+1}^{\pi\left(\sqrt{\frac{x}{p_{i_2}p_{i_3}...p_{i_{k-1}}}}\right)}\left(\pi\left(\frac{x}{p_1p_2...p_{k-1}}\right)-i_1\right)
    \end{equation}
\end{lemma}

\begin{proof}
    We shall prove this by inducting on $k$, we note that \cref{cor:semiprimecor} will serve as the basecase for $k=2$. Now we shall assume that the lemma is true for $k-1$, and then prove it for $k$. 

    We begin by deriving \cref{eq: kalmostprimes nversion}. To see this we shall count the number of $k$-almost primes less than or equal to $x$ whose smallest prime factor is $p_{i_{k-1}}$. Given such a $k$-almost prime $m\leq x$, we can write $m=p_{i_{k-1}}q$ where $q$ is a $k-1$-almost prime. Thus, we have that $p_{i_{k-1}}\leq q\leq \frac{x}{p_{i_{k-1}}}$. Then by the inductive hypothesis the number of such $k-1$ almost primes is given by 
    \[
    \sum_{i_{k-2}=i_{k-1}}^{\pi\left(\sqrt[k-1]{x/p_{i_{k-1}}}\right)}\sum_{i_{k-3}=i_{k-2}}^{\pi\left(\sqrt[k-2]{\frac{x/p_{i_{k-1}}}{p_{i_{k-2}}}}\right)}...\sum_{i_1=i_2}^{\pi\left(\sqrt{\frac{x/p_{i_{k-1}}}{p_{i_2}p_{i_3}...p_{i_{k-2}}}}\right)}\left(\pi\left(\frac{x/p_{i_{k-1}}}{p_1p_2...p_{k-2}}\right)-i_1+1\right).
    \]
    Thus, the number of $k$-almost primes whose smallest prime factor is $p_{i_{k-1}}$ is 
    \[
    \sum_{i_{k-2}=i_{k-1}}^{\pi\left(\sqrt[k-1]{\frac{x}{p_{i_{k-1}}}}\right)}\sum_{i_{k-3}=i_{k-2}}^{\pi\left(\sqrt[k-2]{\frac{x}{p_{i_{k-2}}p_{i_{k-1}}}}\right)}...\sum_{i_1=i_2}^{\pi\left(\sqrt{\frac{x}{p_{i_2}p_{i_3}...p_{i_{k-1}}}}\right)}\left(\pi\left(\frac{x}{p_1p_2...p_{k-1}}\right)-i_1+1\right).
    \]
    Now by doing the same counting argument by observing that the largest prime factor of a $k$-almost prime less than or equal to $x$ will be $\pi(\sqrt[k]{x})$. Thus, to count the number of $k$-almost primes less than $x$ whose smallest possible prime factor is $p_n$, we just add up the add the number of $k$-almost primes whose smallest prime factor is $p_{i_{k-1}}$ for each $i_{k-1}$ from $n$ to $\pi(\sqrt[k]{x})$ which is exactly \cref{eq: kalmostprimes nversion}.

    Now for the square-free case, we remark that the same proof works nearly verbatim. The main difference is that when we write our square-free $k$-almost prime whose smallest prime factor is $p_{i_{k-1}}$ as $n=p_{i_{k-1}}q$, we will have that $q$ must be a square-free $k-1$-almost prime. This will give us that $p_{i_{k-1}}<q\leq \frac{x}{p_{i_{k-1}}}$. Then applying the induction hypothesis and counting up in the same way gives us \cref{eq: sqfrkalmostprime nversion}. We remark that since we have the strict inequality $p_{i_{k-1}}<q$ when applying the inductive hypothesis we must start our sum at $i_{k-2}=i_{k-1}+1$ which is the key difference of the two formulas.
\end{proof}

\raggedright
\bibliographystyle{alpha}
\bibliography{main.bib}
\end{document}